\documentclass[final,reqno,twoside,a4paper,11pt]{amsart}

\usepackage{mltools}
\usepackage{verbatim}
\usepackage{mlmath62,mlthm10-REAR,mllocal}
\usepackage{upref,amsmath,amssymb,amsthm,mathrsfs}
\usepackage{tikz,graphics,latexsym}
\usepackage{times}
\usepackage{amsmath}
\usepackage{amssymb}
\usepackage{amsthm}
\usepackage{amsmath}
\usepackage{amsfonts}
\usepackage{amssymb}
\usepackage[all]{xy}
\usepackage{xcolor}
\usepackage[utf8]{inputenc}

\usepackage[colorlinks=true,linkcolor=blue,citecolor=blue]{hyperref}

\usepackage[scaled]{helvet} 
\usepackage{courier} 
\usepackage[mathbf]{euler}
\usepackage{mllocal} 

\usepackage{pgfplots}

\numberwithin{equation}{section}

\begin{document}
\title {\bf The geometrical information encoded by the Euler obstruction of a map}  

\vspace{1cm}
\author{Nivaldo G. Grulha Jr., Camila Ruiz and Hellen Santana}

\date{\bf }
\maketitle
\begin{center}
{	{\small \textit{Dedicated to Jean-Paul Brasselet on the occasion of his 75th birthday.}}}
\end{center}

\begin{abstract}

In this work we investigate the topological information captured by the Euler obstruction of a map, $f:(X,0)\to (\mathbb{C}^{2},0)$, with $(X,0)$ a germ of a complex $d$-equidimensional singular space, with $d > 2$, and its relation with the local Euler obstruction of the coordinate functions and, consequently, with the Brasselet number. Moreover, under some technical  conditions on the departure variety we relate the Chern number of a special collection related to the map-germ $f$ at the origin with the number of cusps of a generic perturbation of $f$ on a stabilization of $(X,f)$.

\end{abstract}

\section*{Introduction}

The local Euler obstruction is an invariant defined by MacPherson in \cite{MacPherson} as one of the main ingredients in his proof of the Deligne-Grothendieck conjecture about the existence of Chern classes for singular varieties. The local Euler obstruction of $X$ at the origin, where $X$ is a sufficiently small representative of the equidimensional analytic germ $(X,0)$, is denoted by ${Eu}_{X}(0)$. 

After MacPherson's pioneer work, the local Euler obstruction has been studied by many authors. In \cite{LT}, L\^e and Teissier showed that the Euler obstruction is equal to an alternated sum of multiplicities of generic polar varieties of $X$ at $0$. In \cite{BLS}, Brasselet, L\^{e} and Seade proved a Lefschetz type formula for the local Euler obstruction: they related this invariant with the topology of the Milnor fibre on $X$ of a generic linear function. 

In \cite{BMPS}, Brasselet, Massey, Parameswaran and Seade defined a new invariant associated to a function $f$ defined on the variety $X$, called the Euler obstruction of the function and denoted by ${Eu}_{f,X}(0)$. In \cite{STV}, Seade, Tib\u{a}r and Verjovsky showed that, up to sign, the Euler obstruction of a function is equal to the number of critical points of a Morsefication of $f$ lying on the regular part of $X$ which shows that this invariant can be seen as a generalization of the Milnor number (\cite{BMPS,STV}).

Based on the multiplicity formula of L\^e and Teissier, J.-P. Brasselet conjectured that a relative version of the Euler obstruction should satisfy a similar multiplicity formula. In  \cite{DG}, the authors define an invariant called the Brasselet number, denoted by ${B}_{f,X}(0)$, and proved that it satisfies a Lê-Greuel type formula, giving an answer to Brasselet's conjecture. 

When $f$ has isolated singularity, ${B}_{f,X}(0)= {Eu}_{X}(0)-{Eu}_{f,X}(0)$ and in the general case, ${B}_{f,X}(0)= {Eu}_{X}(0)-{D}_{f,X}(0)$, where ${D}_{f,X}(0)$ is the defect defined and studied in Section 6 of \cite{BMPS}. Such defect is related to the L\^e numbers (Proposition 6.1 in \cite{BMPS}), when $X$ is smooth, and to the L\^e-Vogel numbers (Theorem 6.2 in \cite{BMPS}) in the general case. Therefore, considering these relations and the result of Seade, Tib\u{a}r and Verjovsky mentioned above, the defect ${D}_{f,X}(0)$ appears to be a good singular version of the Milnor number, whereas the Brasselet number appears to be a good relative version of the Euler obstruction. 

The Euler obstruction of a map, defined by Grulha \cite{Grulha2008}, and the Chern number of collections of forms, defined by Ebeling and Gusein-Zade \cite{Ebeling2007a}, were defined in the first decade of the 21st century and, in \cite{brasselet2010euler}, Brasselet, Grulha and Ruas related these two invariants. The Euler obstruction of a map is a generalization of the Euler obstruction of a function, defined in \cite{BMPS}.


The  Chern number of collections of forms has a foudantion similar to the Euler obstruction of maps; however it has traveled a different trajectory ``avoiding'' cellular decompotitions, which are, in some sense, covered by the study of the loci of collections of forms.  That is another way to present a generalization of the local Euler obstruction and, in \cite{GaffneyGrulhaJofSing}, Gaffney and Grulha present an algebraic treatment to study the Chern number.


Since the Euler obstruction of a function-germ $g$ at the origin gives important topological information about $g$, more precisely, it counts the number of Morse points on the regular part of $X$ of a generic perturbation of $g$, a natural question has risen: which kind of topological information could be encoded inside the Euler obstruction of a map and how this invariant is related with the Euler obstruction of the coordinate functions of $f$.


 In this paper we relate all these invariants and present geometrical  information that are captured by the Euler obstruction of a map, and by the Chern number, in the context of an application $f: (X,0) \to (\mathbb{C}^{2},0).$  

\begin{center}\textbf{Acknowledgements}\end{center}

The authors are grateful to professor Brasselet, to whom they dedicate this article on the occasion of his 75th birthday. The first author started his research on the local Euler obstruction and its generalizations during his Ph.D. in the Universit\'{e} de la M\'{e}diterran\'{e}e (currently called Aix-Marseille Universit\'{e})  under the advisory of professor Jean-Paul Brasselet. Some questions answered in this work are strong related and raised from inicial questions proposed by Professor Brasselet to the first author during his Ph.D or elaborated by the doctoral jury during his thesis defense in 2007.

The authors also thank professors Dutertre, Gaffney, Ruas and Seade for all productive conversations about Euler obstruction, GSV-index and related topics for their important remarks and suggestions that were crucial to the progress of this work. We also thank professors Ebeling and Zach for the fruitful conversations during the Thematic Program on Singularity Theory, at IMPA, Rio de Janeiro, Brazil, 2020, which were essential in the development of this paper. 

The first author was supported by FAPESP, under grant 2019/21181-02, and by CNPq, grant 303046/2016-3. The second author was partially supported by DMA-UFTM. The third author was supported by FAPESP, grant 2015/25191-9. The authors also thank PROBAL (CAPES-DAAD), grant 88881.198862/2018- 01.

\section{Local Euler obstruction and and generalizations}

\hspace{0,5cm} In this section, we recall  the definition of the local Euler obstruction, the Euler obstruction of a function and the Brasselet number, both generalizations of the local Euler obstruction. The main references for this section are \cite{Ms1,MacPherson,DG, santana2019brasselet}.

Let $(X,0)\subset(\mathbb{C}^n,0)$ be an equidimensional reduced complex analytic germ of dimension $d$ in a open set $U\subset\mathbb{C}^n.$ Consider a complex analytic Whitney stratification $\{V_{\lambda}\}$ of $U$ adapted to $X$ such that $\{0\}$ is a stratum. We choose a small representative of $(X,0),$ denoted by $X,$ such that $0$ belongs to the closure of all strata. We write $X=\cup_{i=0}^{q} V_i,$ where $V_0=\{0\}$ and $V_q$ denotes the regular part $X_{reg}$ of $X.$ We suppose that $V_0,V_1,\ldots,V_{q-1}$ are connected and that the analytic sets $\overline{V_0},\overline{V_1},\ldots,\overline{V_q}$ are reduced. We write $d_i=dim(V_i), \ i\in\{1,\ldots,q\}.$ Note that $d_q=d.$ 

Let $G(d,n)$ be the Grassmannian manifold, $x\in X_{reg}$ and consider the Gauss map $\phi: X_{reg}\rightarrow U\times G(d,n)$ given by $x\mapsto(x,T_x(X_{reg})).$ 

\begin{definition}
The closure of the image of the Gauss map $\phi$ in $U\times G(d,n)$, denoted by $\tilde{X}$, is called \textbf{Nash modification} of $X$. It is a complex analytic space endowed with a holomorphic projection map $\nu:\tilde{X}\rightarrow X.$
\end{definition}

Consider the extension of the tautological bundle $\mathcal{T}$ over $U\times G(d,n).$ Since \linebreak$\tilde{X}\subset U\times G(d,n)$, we consider $\tilde{T}$ the restriction of $\mathcal{T}$ to $\tilde{X},$ called the \textbf{Nash bundle}, and $\pi:\tilde{T}\rightarrow\tilde{X}$ the projection of this bundle.

In this context, denoting by $\varphi$ the natural projection of $U\times G(d,n)$ at $U,$ we have the following diagram:

$$\xymatrix{
\tilde{T} \ar[d]_{\pi}\ar[r] & \mathcal{T}\ar[d] \\ 
\tilde{X}\ar[d]_{\nu}\ar[r] & U\times G(d,n)\ar[d]^{\varphi} \\ 
X\ar[r] & U\subseteq\mathbb{C}^n \\}  $$

Adding the stratum $U \setminus X$ we obtain a Whitney stratification of $U$. Let us denote the restriction to $X$ of the tangent bundle of $U$ by $TU|_{X}$. We know that a stratified vector field $v$ on $X$ means a continuous section of $TU|_{X}$ such that if $x \in V_{i} \cap 
 X$ then $v(x) \in T_{x}(V_{i})$. From Whitney's condition (a), one has the following lemma.

\begin{lemma} \cite{brasselet1981classes}.
	Every stratified vector field $v$, non-null on a subset $A \subset X$, has a canonical lifting to a non-null section $\tilde{v}$ of the Nash bundle $\widetilde {\mathcal{T}}$ over $\nu^{-1}(A) \subset \widetilde{X}$.
\end{lemma}

Now consider a stratified radial vector field $v(x)$ in a neighborhood of $\left\{0\right\}$ in $X$, {\textit{i.e.}}, there is $\varepsilon_{0}$ such that for every $0<\varepsilon \leq \varepsilon_{0}$, $v(x)$ is pointing outwards the ball ${B}_{\varepsilon}$ over the boundary ${S}_{\varepsilon} := \partial{{B}_{\varepsilon}}$.

The following interpretation of the local Euler obstruction has been given by Brasselet and Schwartz in \cite{brasselet1981classes}. As said before, the original definition is presented in \cite{MacPherson}.

\begin{definition}
	Let $v$ be a radial vector field on $X \cap {S}_{\varepsilon}$ and $\tilde{v}$ the lifting of $v$ on $\nu^{-1}(X \cap {S}_{\varepsilon})$ to a section of the Nash bundle. The \textbf{local Euler obstruction} (or simply the \textbf{Euler obstruction}) ${\rm Eu}_{X}(0)$ is defined to be the obstruction to extending $\tilde{v}$ as a nowhere zero section of $\widetilde {\mathcal{T}}$ over $\nu^{-1}(X \cap {B}_{\varepsilon})$.
\end{definition}

More precisely, let $\mathcal{O}{(\tilde{v})} \in \mathbb{H}^{2d}(\nu^{-1}(X \cap 
{B_{\varepsilon}}), \nu^{-1}(X \cap {S_{\varepsilon}}), \mathbb{Z})$ be the 
obstruction cocycle to extending $\tilde{v}$ as a nowhere zero section of 
$\widetilde {\mathcal{T}}$ inside $\nu^{-1}(X\cap {B_{\varepsilon}})$. The 
local Euler obstruction ${\rm Eu}_{X}(0)$ is defined as the evaluation of the 
cocycle $\mathcal{O}(\tilde{v})$ on the fundamental class of the pair $[\nu^{-1}(X 
\cap {B_{\varepsilon}}), \nu^{-1}(X \cap {S_{\varepsilon}})]$ and therefore 
it is an integer.

In \cite{BLS}, Brasselet, Lê and Seade proved a formula to make the calculation of the Euler obstruction easier.

\begin{theorem}(Theorem 3.1 of \cite{BLS})
Let $(X,0)$ and $\mathcal{V}$ be given as before, then for each generic linear form $l,$ there exists $\varepsilon_0$ such that for any $\varepsilon$ with $0<\varepsilon<\varepsilon_0$ and $\delta\neq0$ sufficiently small, the Euler obstruction of $(X,0)$ is equal to 

$$Eu_X(0)=\sum^{q}_{i=1}\chi(V_i\cap B_{\varepsilon}\cap l^{-1}(\delta)).Eu_{X}(V_i),$$

\noindent where $\chi$ is the Euler characteristic, $Eu_{X}(V_i)$ is the Euler obstruction of $X$ at a point of $V_i, \ i=1,\ldots,q$ and $0<|\delta|\ll\varepsilon\ll1.$
\end{theorem} 


Let us give the definition of another invariant introduced by Brasselet, Massey, Parameswaran and Seade in \cite{BMPS}. Let $f:X\rightarrow\mathbb{C}$ be a holomorphic function with isolated singularity at the origin given by the restriction of a holomorphic function $F:U\rightarrow\mathbb{C}$ and denote by $\overline{\nabla}F(x)$ the conjugate of the gradient vector field of $F$ at $x\in U,$ $$\overline{\nabla}F(x):=\left(\overline{\frac{\partial F}{\partial x_1}},\ldots, \overline{\frac{\partial F}{\partial x_n}}\right).$$

Since $f$ has an isolated singularity at the origin, for all $x\in X\setminus\{0\},$ the projection $\hat{\zeta}_i(x)$ of $\overline{\nabla}F(x)$ over $T_x(V_i(x))$ is nonzero, where $V_i(x)$ is a stratum containing $x.$ Using this projection, the authors constructed, in \cite{BMPS}, a stratified vector field on $X,$ denoted by $\overline{\nabla}f(x).$ Let $\tilde{\zeta}$ be the lifting of $\overline{\nabla}f(x)$ as a section of the Nash bundle $\tilde{T}$ over $\tilde{X}$, without singularity on $\nu^{-1}(X\cap S_{\varepsilon}).$

Let $\mathcal{O}(\tilde{\zeta})\in\mathbb{H}^{2n}(\nu^{-1}(X\cap B_{\varepsilon}),\nu^{-1}(X\cap S_{\varepsilon}))$ be the class of the obstruction cocycle for extending $\tilde{\zeta}$ as a non zero section of $\tilde{T}$ inside $\nu^{-1}(X\cap B_{\varepsilon}).$

\begin{definition}
The \textbf{local Euler obstruction of the function} $f, Eu_{f,X}(0)$ is the evaluation of $\mathcal{O}(\tilde{\zeta})$ on the fundamental class $[\nu^{-1}(X\cap B_{\varepsilon}),\nu^{-1}(X\cap S_{\varepsilon})].$
\end{definition}

For an equivalent definition of the Euler obstruction of a function using forms one can see  \cite{DutertreGrulhaJofSing}.

The next theorem compares the Euler obstruction of a space $X$ with the Euler obstruction of function defined on $X.$

\begin{theorem}\label{Euler obstruction of a function formula}(Theorem 3.1 of \cite{BMPS})
Let $(X,0)$ and $\mathcal{V}$ be given as before and let \linebreak$f:(X,0)\rightarrow(\mathbb{C},0)$ be a function with an isolated singularity at $0.$ For $0<|\delta|\ll\varepsilon\ll1,$ we have
 $$Eu_{f,X}(0)=Eu_X(0)-\sum_{i=1}^{q}\chi(V_i\cap B_{\varepsilon}\cap f^{-1}(\delta)).Eu_X(V_i).$$
\end{theorem}

Let us recall some important definitions. Let $\mathcal{V}=\{V_{\lambda}\}$ be a stratification of a reduced complex analytic space $X.$

\begin{definition}
Let $p$ be a point in a stratum $V_{\beta}$ of $\mathcal{V}.$ A \textbf{degenerate tangent plane of $\mathcal{V}$ at $p$} is an element $T$ of the Grassmanian manifold of $n_{\alpha}$-planes in $\mathbb{C}^n$ such that $T=\displaystyle\lim_{p_i\rightarrow p}T_{p_i}V_{\alpha},$ where $p_i\in V_{\alpha}$, $V_{\alpha}\neq V_{\beta}$ and $n_{\alpha}=\dim V_{\alpha}.$
\end{definition}

\begin{definition}
Let $(X,0)\subset(U,0)$ be a germ of complex analytic space in $\mathbb{C}^n$ equipped with a Whitney stratification and let $f:(X,0)\rightarrow(\mathbb{C},0)$ be a holomorphic function, given by the restriction of a holomorphic function $F:(U,0)\rightarrow(\mathbb{C},0).$ Then $0$ is said to be a \textbf{generic point}\index{holomorphic function germ!generic point of} of $f$ if the hyperplane $Ker(d_0F)$ is transverse in $\mathbb{C}^n$ to all degenerate tangent planes of the Whitney stratification at $0.$ 
\end{definition}

Now, let us see the definition of a Morsefication of a function. 

\begin{definition}
Let $\{V_0,V_1,\ldots,V_q\},$ with $0\in V_0,$ a Whitney stratification of the complex analytic space $X.$ A function $f:(X,0)\rightarrow(\mathbb{C},0)$ is said to be \textbf{Morse stratified} if $\dim V_0\geq1, f|_{V_0}: V_0\rightarrow\mathbb{C}$ has a Morse point at $0$ and $0$ is a generic point of $f$ with respect to $V_{i},$ for all $ i\neq0.$
\end{definition}

A \textbf{stratified Morsefication}\index{holomorphic function germ!stratified Morsefication of} of a germ of holomorphic function $f:(X,0)\rightarrow(\mathbb{C},0)$ is a perturbation $\tilde{f}$ of $f$ such that $\tilde{f}$ is Morse stratified.

In \cite{STV}, Seade, Tib\u{a}r and Verjovsky proved that the Euler obstruction of a function $f$ is also related to the number of Morse critical points of a stratified Morsefication of $f.$

\begin{proposition}(Proposition 2.3 of \cite{STV})\label{Eu_f and Morse points}
Let $f:(X,0)\rightarrow(\mathbb{C},0)$ be a germ of holomorphic function with isolated singularity at the origin. Then, \begin{center}
$Eu_{f,X}(0)=(-1)^dn_{reg},$
\end{center}
where $n_{reg}$ is the number of Morse points in $X_{reg}$ in a stratified Morsefication of $f.$
\end{proposition}

Let $X$ be a reduced complex analytic space (not necessarily equidimensional) of dimension $d$ in an open set $U\subseteq\mathbb{C}^n$ and let $f:(X,0)\rightarrow(\mathbb{C},0)$ be a holomorphic map. We write $X^f=X\cap\{f=0\}.$

\begin{definition}\label{good stratification}
A \textbf{good stratification of $X$ relative to $f$} is a stratification $\mathcal{V}$ of $X$ which is adapted to $X^f$ such that $\{V_{\lambda}\in\mathcal{V},V_{\lambda}\nsubseteq X^f\}$ is a Whitney stratification of $X\setminus X^f$ and such that for any pair $(V_{\lambda},V_{\gamma})$ such that $V_{\lambda}\nsubseteq X^f$ and $V_{\gamma}\subseteq X^f,$ the $(a_f)$-Thom condition is satisfied, that is, if $p\in V_{\gamma}$ and $p_i\in V_{\lambda}$ are such that $p_i\rightarrow p$ and $T_{p_i} V(f|_{V_{\lambda}}-f|_{V_{\lambda}}(p_i))$ converges to some $\mathcal{T},$ then $T_p V_{\gamma}\subseteq\mathcal{T}.$
\end{definition}



Let $\mathcal{V}$ be a good stratification of $X$ relative to $f.$

\begin{definition}
The \textbf{critical locus of $f$ relative to $\mathcal{V}$}, $\Sigma_{\mathcal{V}}f,$ is given by the union \begin{center}$\Sigma_{\mathcal{V}}f=\displaystyle\bigcup_{V_{\lambda}\in\mathcal{V}}\Sigma(f|_{V_{\lambda}}).$\end{center}
\end{definition}

Let $g:(X,0)\rightarrow(\mathbb{C},0)$ be a holomorphic function-germ.

\begin{definition}
If $\mathcal{V}=\{V_{\lambda}\}$ is a stratification of $X,$ the \textbf{relative polar variety of $f$ and $g$ with respect to $\mathcal{V}$}, denoted by $\Gamma_{f,g}(\mathcal{V}),$ is the the union $\cup_{\lambda}\Gamma_{f,g}(V_{\lambda}),$ where $\Gamma_{f,g}(V_{\lambda})$ denotes the closure in $X$ of the critical locus of $(f,g)|_{V_{\lambda}\setminus X^f}.$  
\end{definition}

\begin{definition}
If $\mathcal{V}=\{V_{\lambda}\}$ is a stratification of $X,$ the \textbf{symmetric relative polar variety of $f$ and $g$ with respect to $\mathcal{V}$}, $\tilde{\Gamma}_{f,g}(\mathcal{V}),$ is the union $\cup_{\lambda}\tilde{\Gamma}_{f,g}(V_{\lambda}),$ where $\tilde{\Gamma}_{f,g}(V_{\lambda})$ denotes the closure in $X$ of the critical locus of $(f,g)|_{V_{\lambda}\setminus (X^f\cup X^g)},$  $X^f=X\cap \{f=0\}$ and $X^g=X\cap \{g=0\}. $ 
\end{definition}

\begin{definition}\label{definition prepolar}
Let $\mathcal{V}$ be a good stratification of $X$ relative to a function
$f:(X,0)\rightarrow(\mathbb{C},0).$ A function $g :(X, 0)\rightarrow(\mathbb{C},0)$ is \textbf{prepolar with respect to $\mathcal{V}$ at the origin} if the origin is a stratified isolated critical point, that is, $0$ is an isolated point of $\Sigma_{\mathcal{V}}g.$
\end{definition}

\begin{definition}\label{definition tractable}
A function $g :(X, 0)\rightarrow(\mathbb{C},0)$ is \textbf{tractable at the origin with respect to a good stratification $\mathcal{V}$ of $X$ relative to $f :(X, 0)\rightarrow(\mathbb{C},0)$} if $dim_0 \ \tilde{\Gamma}^1_{f,g}(\mathcal{V})\leq1$ and, for all strata $V_{\alpha}\subseteq X^f$,
$g|_{V_{\alpha}}$ has no critical point in a neighbourhood of the origin except perhaps at the origin itself.

\end{definition}




We recall the definition of the Brasselet number introduced by Dutertre and Grulha \cite{DG}. Let $f: (X,0)\rightarrow(\mathbb{C},0)$ be a germ of holomorphic function, $X$ a suficiently small representative of the germ, and let $\mathcal{V}$ be a good stratification of $X$ relative to $f.$ We denote by $V_1,\ldots, V_q$ the strata of $\mathcal{V}$ that are not contained in $\{f=0\}$ and we assume that $V_1,\ldots, V_{q-1}$ are connected and that $V_{q}=\linebreak X_{reg}\setminus \{f=0\}.$ Note that $V_q$ could be not connected.  
 
\begin{definition}
Suppose that $X$ is equidimensional. Let $\mathcal{V}$ be a good stratification of $X$ relative to $f.$ The \textbf{Brasselet number} of $f$ at the origin, $B_{f,X}(0),$ is defined by \begin{center}
$B_{f,X}(0)=\sum_{i=1}^{q}\chi(V_i\cap f^{-1}(\delta)\cap B_{\varepsilon})Eu_X(V_i),$
\end{center}
where $0<|\delta|\ll\varepsilon\ll1.$
\end{definition} 

\noindent\textbf{Remark:} If $V_q^i$ is a connected component of $V_{q},$ $Eu_X(V_q^i)=1.$

Notice that if $f$ has a stratified isolated singularity at the origin, then \linebreak$B_{f,X}(0)=Eu_{X}(0)-Eu_{f,X}(0)$ (see Theorem \ref{Euler obstruction of a function formula}).

In \cite{DG}, Dutertre and Grulha proved formulas describing the topological relation between the Brasselet number and a number of certain critical points of a special type of perturbation of functions. Let us now present some of these results. First we need the definition of a special type of Morsefication, introduced by Dutertre and Grulha.

\begin{definition}
A \textbf{partial Morsefication} of $g:f^{-1}(\delta)\cap X\cap B_{\varepsilon}\rightarrow\mathbb{C}$ is a function $\tilde{g}: f^{-1}(\delta)\cap X\cap B_{\varepsilon}\rightarrow\mathbb{C}$ (not necessarily holomorphic) which is a local Morsefication of all isolated critical points of $g$ in $f^{-1}(\delta)\cap X\cap \{g\neq 0\}\cap B_{\varepsilon}$ and which coincides with $g$ outside a small neighborhood of these critical points.
\end{definition} 

Let $g : (X, 0) \rightarrow (\mathbb{C},0)$ be a holomorphic function which is tractable at the origin with respect to $\mathcal{V}$ relative to $f .$ Then $\tilde{\Gamma}_{f,g}$ is a complex analytic curve and for $0<|\delta|\ll1$ the critical points of $g|_{f^{-1}(\delta)\cap X}$ in $B_{\varepsilon}$ lying outside $\{g=0\}$ are isolated.
Let $\tilde{g}$ be a partial Morsefication of $g:f^{-1}(\delta)\cap X\cap B_{\varepsilon}\rightarrow\mathbb{C}$ and, for each \linebreak$i\in\{1,\ldots, q\},$ let $n_i$ be the number of stratified Morse critical points of $\tilde{g}$ appearing on $V_i\cap f^{-1}(\delta)\cap \{g\neq 0\}\cap B_{\varepsilon}.$

\begin{theorem}\label{Corollary 4.3 of DG}(Corollary 4.3 of \cite{DG})
Suppose that $X$ is equidimensional and that $g$ is tractable at the origin with respect to $\mathcal{V}$ relative to $f.$ For $0<|\delta|\ll\varepsilon\ll 1,$ we have\begin{center}
 $\chi(X\cap f^{-1}(\delta)\cap B_{\varepsilon},Eu_X)-\chi(X\cap g^{-1}(0)\cap f^{-1}(\delta)\cap B_{\varepsilon},Eu_X)=(-1)^{d-1}n_q.$
 \end{center}
\end{theorem}

If one supposes, in addition, that $g$ is prepolar, a consequence of this result is a Lê-Greuel type formula for the Brasselet number.

\begin{theorem}\label{4.4 DG}(Theorem 4.4 of \cite{DG})
Suppose that $X$ is $d$-equidimensional and that $g$ is prepolar with respect to $\mathcal{V}$ at the origin. For $0<|\delta|\ll\varepsilon\ll 1,$ we have\begin{center}
 $B_{f,X}(0)-B_{f,X^g}(0)=(-1)^{d-1}n_q,$
 \end{center}
 where $n_q$ is the number of stratified Morse critical points on the top stratum $V_q\cap f^{-1}(\delta)\cap B_{\varepsilon}$ appearing in a Morsefication of $g:X\cap f^{-1}(\delta)\cap B_{\varepsilon}\rightarrow \mathbb{C}.$
\end{theorem}





\hspace{0,5cm} In \cite{santana2019brasselet}, the author generalized the last formula to the case where the critical locus of $g,$ $\Sigma_{\mathcal{V}} g,$ is one-dimensional and that $\Sigma_{\mathcal{V}} g\cap \{f=0\}=\{0\}.$ We can decompose $\Sigma_{\mathcal{V}}g$ into branches $b_j, \Sigma_{\mathcal{V}} g=\bigcup_{\alpha=1}^q\Sigma g|_{V_{\alpha}}\cup\{0\}=b_1\cup\ldots\cup b_r,$ where $b_j\subseteq V_{\alpha},$ for some $\alpha\in\{1,\ldots,q\}.$  Let $\delta$ be a regular value of $f, 0<|\delta|\ll1,$ and let us write, for each $j\in\{1,\ldots,r\}, f^{-1}(\delta)\cap b_j=\{x_{i_1},\ldots,x_{i_{k(j)}}\}$ and $m_{f,b_j}=k(j).$  Let $\varepsilon$ be sufficiently small such that the local Euler obstruction of $X$ and of $X^g$ are constant on $b_j\cap B_{\varepsilon}$. In this case, we denote by $Eu_{X}(b_j)$ (respectively, $Eu_{X^g}(b_j)$) the local Euler obstruction of $X$ (respectively, $X^g$) at a point of $b_j\cap B_{\varepsilon}.$

\begin{theorem}\label{Generalization 6.4 of DG}
Suppose that $g$ is tractable at the origin with respect to  $\mathcal{V}$ relative to $f.$ Then, for $0<|\delta|\ll\varepsilon\ll1,$ 
\begin{center}
$B_{f,X}(0)-B_{f,X^g}(0)-\sum_{j=1}^{r}m_{f,b_j}(Eu_X(b_j)-Eu_{X^g}(b_j))=(-1)^{d-1}n_q,$
\end{center}
where $n_q$ is the number of stratified Morse critical points of a partial Morsefication of \linebreak$g:X\cap f^{-1}(\delta)\cap B_{\varepsilon}\rightarrow\mathbb{C}$ appearing on $X_{reg}\cap f^{-1}(\delta)\cap \{g\neq 0\}\cap B_{\varepsilon}.$ 
\end{theorem} 

\section{Relations between the Euler obstruction of a Map, of a Function and the Brasselet number}

The notion of local Chern obstruction extends the notion of local Euler obstruction in the case of collections of germs of 1-forms. More precisely, Ebeling and Gusein-Zade perform the following construction.

Let $(X, 0)\subset(\mathbb{C}^{n},0)$ be the germ of a $d$-equidimensional reduced complex analytic variety at the origin. Let $\{\omega_j^{(i)}\}$ be a collection of germs of 1-forms on $(\mathbb{C}^{n},0)$ such that $i=1,\ldots, s$; $j=1, \ldots, d-k_{i}+1$, where the $k_{i}$ are non-negative integers with $\sum_{i=1}^{s}k_{i}=d$. Let $\varepsilon>0$ be small enough so that there is a representative $X$ of the germ $(X, 0)$ and representatives $\{\omega_{j}^{(i)}\}$ of the germs of 1-forms inside the ball $B_{\varepsilon}(0)\subset\mathbb{C}^{n}.$

\begin{definition}
	For a fixed $i$, the locus of the subcollection $\{\omega_{j}^{(i)}\}$ is the set of points $x\in X$ such that there exists a sequence $x_{n}$ of points from the non-singular part $X_{{\rm{reg}}}$ of the variety $V$ such that the sequence $T_{x_{n}}X_{{\rm{reg}}}$ of the tangent spaces at the points $x_{n}$ has a limit $L$ (in $G(d,n)$) and the restrictions of the 1-forms $\omega_{1}^{(i)}, \ldots, \omega_{d-k_{i}+1}^{(i)}$ to the subspace $L\subset T_{x}\mathbb{C}^{n}$ are linearly dependent.
\end{definition}

\begin{definition}\label{specialpoint} A point $x\in X$ is called a {\it special point} of the collection $\{\omega_{j}^{(i)}\}$ if it is in the intersection of the loci of the subcollections $\{\omega_{j}^{(i)}\}$ for each $i=1,\ldots, s$. The collection $\{\omega_{j}^{(i)}\}$ of 1-forms has an isolated special point at $\{0\}$ if it has no special point on $X$ in a punctured neighbourhood of the origin.
\end{definition}

Note that in Definition \ref{specialpoint}, for each fixed $i$, we require each subcollection $\{\omega_{j}^{(i)}\}$ to be linearly dependent when restricted to the same limit plane. Also note that if an element of the collection has less than maximal rank at a point, then it is linearly dependent on all planes passing through the point.

Let $\{\omega_{j}^{(i)}\}$ be a collection of germs of 1-forms on $(X,0)$ with an isolated special point at the origin. Let $\nu:\tilde{X}\rightarrow X$ be the Nash transformation of the variety $X$ and $\tilde{T}$ be the Nash bundle. The collection of 1-forms $\{\omega_{j}^{(i)}\}$ gives rise to a section $\Gamma(\omega)$ of the bundle
$$\tilde{\mathbb{T}}=\bigoplus_{i=1}^{s}\bigoplus_{j=1}^{d-k_{i}+1}\tilde{T}_{i,j}^{*},$$ where $\tilde{T}_{i,j}^{*}$ are copies of the dual Nash bundle $\tilde{T}^{*}$ over the Nash transformation $\tilde{X}.$

Let $\mathbb{D}\subset\tilde{\mathbb{T}}$ be the set of pairs $(x,\{\alpha_{j}^{(i)}\})$ where $x\in\tilde{X}$ and the collection of 1-forms $\{\alpha_{j}^{(i)}\}$ is such that $\alpha_{1}^{(i)},\ldots, \alpha_{n-k_{i}+1}^{(i)}$ are linearly dependent for each $i=1, \ldots, s.$

\begin{definition}
	Let $0$ be a special point of the collection $\{\omega_{j}^{(i)}\}$. The local Chern obstruction ${\rm{Ch}}_{X,0}\{\omega_{j}^{(i)}\}$ of the collection of germs of 1-forms $\{\omega_{j}^{(i)}\}$ on $(X,0)$ at the origin is the obstruction to extend the section $\Gamma(\omega)$ of the fibre bundle $\tilde{\mathbb{T}}\backslash\mathbb{D}\rightarrow \tilde{X}$ from $\nu^{-1}(X\cap S_{\varepsilon})$ to $\nu^{-1}(X\cap B_{\varepsilon})$.
\end{definition}

We can see that we have the correct obstruction dimension as follows. For each $\tilde{\mathbb{T}}_i=\bigoplus_{j=1}^{d-k_{i}+1}\tilde{T}_{i,j}^{*}$ let $\mathbb{D}_{i}\subset\tilde{\mathbb{T}}_{i}$ be the set of pairs $(x,\{\alpha_{j}^{(i)}\})$, where $x\in\tilde{X}$ and the collection $\{\alpha_{j}^{(i)}\}$ is such that $\alpha_{1}^{(i)}, \ldots,\alpha_{d-k_{i}+1}^{(i)}$ are linearly dependent. Then, the set $\tilde{\mathbb{T}}_i\backslash\mathbb{D}_{i}$ is a Stiefel manifold, with associated obstruction dimension equal to $k_i$, and therefore the obstruction dimension for $\tilde{\mathbb{T}}$ is $\sum_{i=1}^{s} k_i=d$.


The following result is a consequence of \cite[Proposition 3.3]{Ebeling2007a}.

\begin{proposition}\rm{
		Let $(X,0)\subset(\mathbb{C}^{n},0)$ be the germ of a $d$-equidimensional reduced complex analytic variety at the origin. Let $\{\omega_{j}^{(i)}\}$ be a collection of germs of $1$-forms on $(\mathbb{C}^{n},0)$ such that $i=1, \ldots, s$; $j=1, \ldots, d-k_{i}+1,$ where the $k_i$ are non-negative integers with $\sum_{i=1}^{s}k_{i}=d$. Let $0$ be an isolated special point for the collection. If $\omega^{(i)}$, $i=2, \ldots, s$, are generic collections of linear forms, then the number ${\rm{Ch}}_{X,0}\{\omega_{j}^{(i)}\}$ does not depend on the choice of the subcollections $\omega^{(i)}$, $i=2, \ldots, s$.}
\end{proposition} 

In in \cite{Grulha2008} Grulha defined a notion of the Euler obstruction of a map, but in his construction the obstruction depends on a certain cellular decomposition. Some years latter, in \cite{brasselet2010euler} Brasselet, Grulha and Ruas compared the notion of the Euler obstruction of a map and the Chern number. In that paper they also proved that the Euler obstruction does not depend on a generic choice in its constuction. Based on this, we define the Euler obstruction of a map in terms of collection of forms.

\begin{definition}
	Let $X$ be an equidimensional complex variety of dimension $d$, $f:(X,0)\to \mathbb{C}^{p}$, a holomorphic map, with $0\leq p \leq d$ and $\omega_{1}=\{df_{1}, df_{2},...,df_{p}\}$, with $df_{i}$ the differential of the coordinate functions of $f$, and $\omega_{2}$ a generic collection, in such way that $0$ is a special point of the collection of collections $\omega = \{\omega_{1},\omega_{2}\}$. We define the Euler obstruction of the map $f$ at the origin, denoted by $Eu
^{*}_{X,f}(0)={\rm{Ch}}_{X,0}\{\omega_{j}^{(i)}\}$.

\end{definition}




\section{The Euler obstruction of a map and Morse critical points}

Consider a map-germ $f:(X,0)\rightarrow(\mathbb{C}^2,0), f(x)=(f_1(x),f_2(x)),$ given as the restriction of a holomorphic map $F:(U,0)\rightarrow(\mathbb{C}^2,0), F(x)=(F_1(x),F_2(x)),$  where $X\subset U\subset\mathbb{C}^n$ has dimension $d>2.$ We aim to compare the local  Euler obstruction of the map $f$ and the Brasselet number of its coordinate functions. 




Let $\mathcal{V}$ be a good stratification of $X$ relative to $f_2,$ which always exists (see \cite{Ms1}). Consider the collection $\{\omega^{(1)},\omega^{(2)}\}$ of $1$-forms, 
where $\{\omega^{(1)}\}=\{df_1,df_2\}$ and $\{\omega^{(2)}\}=\{l_1,\ldots, l_{d}\}$ is a generic collection of linear forms \cite{ebeling2005radial}. Notice that, by \cite{brasselet2010euler}, $Eu^{*}_{f,X}(0)=Ch_{X,0}\{\omega^{(i)}_j\}.$ Hence,  the number $Ch_{X,0}\{\omega^{(i)}_j\}$ depends only on the map-germ $f.$ 

From now on, we assume that the collection $\{\omega^{(1)},\omega^{(2)}\}$ of $1$-forms has an isolated special point at the origin.

\begin{lemma}\label{correspondence between special and critical points}
	 Suppose that $f_1$ is tractable at the origin with respect to the good stratification $\mathcal{V}$ of $X$ relative to $f_2$. Then the Euler obstruction $Eu^{*}_{f, X\cap f_2^{-1}(\delta)}(x_0)$ of $f$ at a point $x_0\in X\cap f_2^{-1}(\delta)$ is equal to the number of Morse critical points of a Morsefication of $f_{1}|_{X_{reg}\cap f_{2}^{-1}(\delta)}$ appearing in $X_{reg}\cap f_2^{-1}(\delta),$ for $0<|\delta|\ll\varepsilon\ll1.$
\end{lemma}
\noindent\textbf{Proof.} 
Let $p$ be a nondegenerate Morse critical point of a perturbation $f_1|_{X_{reg}\cap f_{2}^{-1}(\delta)}+tL$ of $f_1|_{X_{reg}\cap f_{2}^{-1}(\delta)}$, where $L$ is a generic linear form and $0<|t|\ll\varepsilon.$ Then $d_pf_1|_{T_pX_{reg}\cap f_{2}^{-1}(\delta)}+tL(p)=0$, which implies that $p$ is in the locus of a generic perturbation of $\{df_1|_{T_pX_{reg}\cap f_{2}^{-1}(\delta)},df_2|_{T_pX_{reg}\cap f_{2}^{-1}(\delta)}\}.$ Since $Ch_{X\cap f_{2}^{-1}(\delta),x_0}\{\omega^{(i)}_j\}$ does not depend on the (generic) choice of the subcollection $\omega^{(2)},$ $p$ is one of the points counted by this Chern obstruction.

Conversely, let $p\in X_{reg}\cap f_{2}^{-1}(\delta)$ be a nondegenerated special point of a small perturbation  $\{\omega^{(i)}_j+tl^{(i)}_j\}$ of $\{\omega^{(i)}_j\},$ where $0<|t|\ll\varepsilon$ and $l^{(i)}_j$ are generic linear forms.  Then $\{df_1|_{T_pX_{reg}\cap f_{2}^{-1}(\delta)}+tl^{(1)}_1,df_2|_{T_pX_{reg}\cap f_{2}^{-1}(\delta)}+tl^{(1)}_2\}$ is linearly dependent. Hence, there exist $\lambda_1,\lambda_2\in\mathbb{C}$ such that \begin{eqnarray*}
 df_2|_{T_pX_{reg}\cap f_{2}^{-1}(\delta)}&=&\frac{\lambda_1}{\lambda_2}df_1|_{T_pX_{reg}\cap f_{2}^{-1}(\delta)}+\frac{t\lambda_1}{\lambda_2}l^{(1)}_1-tl^{(1)}_2\\
 &=&\lambda df_1|_{T_pX_{reg}\cap f_{2}^{-1}(\delta)}+tl,
\end{eqnarray*}
\noindent where $l$ is the linear form $\frac{\lambda_1}{\lambda_2}l^{(1)}_1-l^{(1)}_2.$ This implies that $p$ is a critical point of $f_{1}|_{X_{reg}\cap f_2^{-1}(\delta)}+tl,$ where $ 0<|\delta|\ll\varepsilon$. Therefore, $p$ is a nondegenerated critical point of a perturbation of $f_{1}|_{X_{reg}\cap f_2^{-1}(\delta)}$ (see Theorem 3.2 of \cite{Ms1}).

Since the classes of perturbations of the collection of forms $\{\omega^{(i)}_j\}$ and of the function $f_{1}|_{X_{reg}\cap f_2^{-1}(\delta)}$ are the same, we conclude that $Ch_{X\cap f_2^{-1}(\delta),x_0}\{\omega^{(i)}_j\}$ is equal to the number $n_{reg}$ of Morse critical points of a Morsefication of $f_{1}|_{X_{reg}\cap f_{2}^{-1}(\delta)}$ appearing in $X_{reg}\cap f_2^{-1}(\delta).$\qed

\begin{proposition}\label{obstrucao de Euler da aplicacao e Numero de Brasselet}
	If $f_1$ is prepolar at the origin with respect to the good stratification $\mathcal{V}$ of $X$ relative to $f_2,$ then 
	$Eu^{*}_{f,X\cap f_2^{-1}(\delta)}(x_0)=(-1)^{d-1}(B_{f_2,X}(0)-B_{f_2,X\cap f_1^{-1}(0)}(0)),$ where $x_0\in X\cap f_2^{-1}(\delta)$
\end{proposition}
\noindent\textbf{Proof.} If $f_1$ is prepolar at the origin with respect to $\mathcal{V},$ $f_1$ is tractable at the origin with respect to $\mathcal{V}.$ Hence, applying Lemma \ref{correspondence between special and critical points} and Theorem \ref{4.4 DG}, we obtain \begin{eqnarray}
Eu^{*}_{f,X\cap f_2^{-1}(\delta)}(x_0)=n_{reg}=(-1)^{d-1}(B_{f_2,X}(0)-B_{f_2,X\cap f_1^{-1}(0)}(0)),
\end{eqnarray}
\noindent where $n_{reg}$ is the number of Morse critical points of a Morsefication of $f_{1}|_{X_{reg}\cap f_{2}^{-1}(\delta)}$ appearing in $X_{reg}\cap f_2^{-1}(\delta).$\qed

A consequence of the last result is a relation between the Euler obstruction of a map and the Euler obstruction of a function.

\begin{corollary}
	Suppose that $f_2$ has an isolated singularity at the origin. If $f_1$ is prepolar at the origin with respect to the good stratification $\mathcal{V}$ of $X$ induced by $f_2$ and $x_0\in X\cap f_2^{-1}(\delta),$ then 
	$Eu^{*}_{f,X\cap f_2^{-1}(\delta)}(x_0)=(-1)^{d-1}(Eu_{f_2,X\cap f_1^{-1}(0)}(0)-Eu_{f_2,X}(0)).$
\end{corollary}
\noindent\textbf{Proof.} If $f_2$ has an isolated singularity at the origin, then $B_{f_2,X}(0)=Eu_{X}(0)-Eu_{f_2,X}(0)$ and $B_{f_2,X\cap f_1^{-1}(0)}(0)=Eu_{X\cap f_1^{-1}(0)}(0)-Eu_{f_2,X\cap f_1^{-1}(0)}(0).$ Since $f_1$ is prepolar at the origin with respect to $\mathcal{V},$ then $ Eu_{X\cap f_1^{-1}(0)}(0)=Eu_{X}(0).$ Hence \begin{eqnarray*}
B_{f_2,X\cap f_1^{-1}(0)}(0)-B_{f_2,X}(0)&=&Eu_{X\cap f_1^{-1}(0)}(0)-Eu_{f_2,X\cap f_1^{-1}(0)}(0)\\
&-&Eu_{X}(0)+Eu_{f_2,X}(0)\\
&=&Eu_{f_2,X}(0)-Eu_{f_2,X\cap f_1^{-1}(0)}(0).
\end{eqnarray*}
\qed

\vspace{0,1cm}

In the last result of this section, we apply formulas of  \cite{santana2019brasselet} to compare the Brasselet number and the Euler obstruction of a map.

\begin{proposition}
	 If $f_1$ is tractable at the origin with respect to the good stratification $\mathcal{V}$ of $X$ relative to $f_2,$ $\dim_0\Sigma_{\mathcal{V}}f_1=1$, $\Sigma_{\mathcal{V}}f_1\cap\{f_2=0\}=\{0\}$ and $x_0\in X\cap f_2^{-1}(\delta),$ then
	
	$$B_{f_2,X}(0)=\sum_{\alpha}\chi(V_{\alpha}\cap f^{-1}(z_0)\cap B_{\varepsilon})Eu_X(V_{\alpha})+(-1)^{d-1}Eu^{*}_{f,X\cap f_2^{-1}(\delta)}(x_0).$$

\end{proposition}

\noindent\textbf{Proof.} Let $z_0=(\alpha,\delta)\in\mathbb{C}^2, 0<|z_0|\ll\varepsilon\ll1,$ be a regular value of $f.$ Following \cite{santana2019brasselet}, since $f_1$ is tractable at the origin with respect to $\mathcal{V},$ 
\begin{eqnarray*}
	\sum_{\alpha}\chi(V_{\alpha}\cap f^{-1}(z_0)\cap B_{\varepsilon})Eu_X(V_{\alpha})&=&B_{f_2,X^{f_1}}(0)-\sum_{j=1}^rm_{f,b_j}(Eu_{X^{f_1}}(b_j)-Eu_X(b_j)).
\end{eqnarray*}

Also, by Theorem 3.2 of \cite{santana2019brasselet}, \begin{eqnarray*}
	B_{f_2,X^{f_1}}(0)=B_{f_2,X}(0)+\sum_{j=1}^rm_{f,b_j}(Eu_{X^{f_1}}(b_j)-Eu_{X}(b_j))-(-1)^{d-1}n_{reg},
\end{eqnarray*}
\noindent where $n_{reg}$ denotes the number of stratified Morse critical points of a partial Morsefication of $f_{1}|_{X\cap f_2^{-1}(\delta)}$ appearing in $X_{reg}\cap f_2^{-1}(\delta).$

Hence, \begin{eqnarray*}
	B_{f_2,X}(0)=\sum_{\alpha}\chi(V_{\alpha}\cap f^{-1}(z_0)\cap B_{\varepsilon})Eu_X(V_{\alpha})+(-1)^{d-1}n_{reg}.
\end{eqnarray*}
Using Lemma \ref{correspondence between special and critical points}, we obtain the formula.\qed

\section{Applications}

\subsection{Lê-Iomdine formula for the Euler obstruction of a map}

Let $f$ be a function with a one-dimensional critical set defined over an open subset of $\mathbb{C}^n$ and $l$ a generic linear form on $\mathbb{C}^n$. In \cite{iomdin1974complex}, Iomdin gave an algebraic relation between the Euler characteristic of the Milnor fibre of $f$ and the Euler characteristic of the Milnor fibre of $f+l^N,$ where $N$ is sufficiently large non-negative integer number. In \cite{le1980ensembles}, Lê proved this same relation in a more geometrical approach and with a way to obtain the Milnor fibre of $f$ by attaching a certain number of $n$-cells to the Milnor fibre of $f|_{\{l=0\}}.$ 

In \cite{Ms2}, Massey worked with a function $f$ with critical locus of higher dimension defined over a nonsingular space and defined the Lê numbers and cycles, which provides a way to numerically describe the Milnor fibre of this function with nonisolated singularity. In that work, Massey presented a relation between the Lê numbers of $f$ and $f+l^N,$ obtaining a Lê-Iomdin type formula between these numbers.

We keep the previous hypothesis about $f_1$ and we suppose, in addition, that  $f_2$ has an isolated singularity at the origin. Consider the perturbation $\tilde{f_1}=f_1+f_2^N$ of $f_1$, where $N>>1$ is an integer.  In \cite{santana2019local}, the author proved that, if $N>>1$ is  sufficiently large and $\delta$ is a regular value of $f_2$,  the number $n_{reg}$ of Morse critical points of a stratified Morsefication of $f_1|_{X\cap f_2^{-1}(\delta)}$ appearing on $X_{reg}\cap f_2^{-1}(\delta)\cap B_{\varepsilon}$ and the number $\tilde{n_{reg}}$ of Morse critical points of a stratified Morsefication of $\tilde{f_1}|_{X\cap f_2^{-1}(\delta)}$ appearing on $X_{reg}\cap f_2^{-1}(\delta)\cap B_{\varepsilon}$ can be compared  as follows \begin{eqnarray}\label{cor 3.12}
\tilde{n_{reg}}=n_{reg}+(-1)^{d-1}\sum_{i=1}^rm_{f_2,b_j}Eu_{f_1,X\cap f_2^{-1}(\delta)}(b_j).
\end{eqnarray} 

In that paper, it was also proved that the number $m_{reg}$ of Morse critical points of a stratified Morsefication of $f_2|_{X\cap f_1^{-1}(\alpha)}$ appearing on $X_{reg}\cap f_1^{-1}(\alpha)\cap B_{\varepsilon}$ and the number $\tilde{m_{reg}}$ of Morse critical points of a stratified Morsefication of $f_2|_{X\cap \tilde{f_1}^{-1}(\alpha^{\prime})}$ appearing on $X_{reg}\cap \tilde{f_1}^{-1}(\alpha^{\prime})\cap B_{\varepsilon}$ can be compared as follows \begin{eqnarray}\label{lemma 4.3}
\tilde{m_{reg}}=m_{reg}+(-1)^{d-1}N\sum_{i=1}^rm_{f_2,b_j}Eu_{f_2,X\cap \tilde{f_1}^{-1}(\alpha^{\prime})}(b_j),
\end{eqnarray} 
\noindent where $ 0<|\alpha|,|\alpha^{\prime}|\ll|\delta|\ll\varepsilon\ll1, \alpha$ is a regular value of $f_1$ and $\alpha^{\prime}$ is a regular value of $\tilde{f_1}.$
Applying these formulas and the results proved in the last section, we obtain a Lê-Iomdine type formula for the Euler obstruction of a map.   We aim to compare the Euler obstruction of the map $f=(f_1,f_2)$ with the Euler obstruction of a perturbation $\tilde{f}=(\tilde{f_1},f_2).$ 

\begin{corollary}
	For a sufficiently large integer number $N>>1,$ if $f_1$ is tractable at the origin with respect to $\mathcal{V}$ and  $x_0\in X\cap f_2^{-1}(\delta),$ then $$Eu^{*}_{\tilde{f},X\cap f_2^{-1}(\delta)}(x_0)=Eu^{*}_{f, X\cap f_2^{-1}(\delta)}(x_0)+(-1)^{d-1}\sum_{i=1}^rm_{f_2,b_j}Eu_{f_1,X\cap f_2^{-1}(\delta)}(b_j).$$
\end{corollary}
\noindent\textbf{Proof.}
Using Lemma \ref{correspondence between special and critical points}, $Eu^{*}_{\tilde{f},X\cap f_2^{-1}(\delta)}(x_0)=\tilde{n_{reg}}$ and $Eu^{*}_{f,X\cap f_2^{-1}(\delta)}(x_0)=n_{reg}.$ Applying these equalities to Equation \ref{cor 3.12}, we obtain the  equality claimed.\qed

\vspace{0,2cm}

We can also compare the Euler obstruction of the map $g=(f_2,f_1)$ with the Euler obstruction of a perturbation $\tilde{g}=(f_2, \tilde{f_1}).$ 

\begin{corollary}
	For a sufficiently large integer number $N>>1,$ if $f_1$ is tractable at the origin with respect to $\mathcal{V}$ and  $x_0\in X\cap f_2^{-1}(\delta),$ then $$Eu^{*}_{\tilde{g},X\cap \tilde{f_1}^{-1}(\alpha)}(x_0)=Eu^{*}_{g, X\cap f_1^{-1}(\alpha)}(x_0)+N(-1)^{d-1}\sum_{i=1}^rm_{f_2,b_j}Eu_{f_2,X\cap \tilde{f_1}^{-1}(\alpha)}(b_j),$$ where $ 0<|\alpha|\ll\varepsilon\ll1.$
\end{corollary}
\noindent\textbf{Proof.}
By Lemma \ref{correspondence between special and critical points}, $Eu^{*}_{\tilde{g},X\cap \tilde{f_1}^{-1}(\alpha^{\prime})}(x_0)=\tilde{m_{reg}}$ and $Eu^{*}_{g,X\cap {f_1}^{-1}(\alpha)}(x_0)=m_{reg}.$ Applying these equalities to Equation \ref{lemma 4.3}, we obtain the equality claimed.\qed

\subsection{Euler obstruction of a map, Chern obstruction and number of cusps}

Applying results of \cite{BOT}, we present in this section an application that generalizes Proposition 3.2 of \cite{brasselet2010euler}, computing the number of cusps of a stabilization $f_{s}:X_{s}\to \mathbb{C}^{2}$. We also present the relation between the multiplicity of the singular set of the application and the Euler obstruction of a map as follows.

\begin{proposition} Let $f:(X,0)\rightarrow(\mathbb{C}^2,0), f(x)=(f_1(x),f_2(x)),$ be a holomorphic map-germ. Consider the collection $\{\omega^{(1)},\omega^{(2)}\}$ of $1$-forms on $\mathbb{C}^n$, 
where $\{\omega^{(1)}\}=\{df_1,df_2\}$ and $\{\omega^{(2)}\}=\{l_1,\ldots, l_{d}\}$ are generic. Assume that the collection $\{\omega^{(1)},\omega^{(2)}\}$ of $1$-forms has an isolated special point at the origin. Then, $$Eu^{*}_{f,X}(0)=\# \ \Sigma(\tilde{f})\cap H,$$ where $\Sigma(\tilde{f})$ is the singular set of a generic perturbation of $f$ and  $H$ is a $\ell$-plane in $\mathbb{C}^n$ in general position with $\Sigma(\tilde{f})$ and $\ell$ is the codimension of $\Sigma(\tilde{f})$.
\end{proposition}

\begin{proof} By Corollary 1 and Proposition 2 of \cite{Ebeling2007a} and by the relation between the Chern number and the Euler obstruction of a map \cite{brasselet2010euler}, we have
\begin{equation*}
  \begin{array}{ccl}
   Eu^{*}_{f,X}(0)  &=&  Ch_{X,0}\{\omega_j^{(i)}\}  \\
      &=&  \# \{\text{special points of } \{\tilde{\omega}_j^{(i)}\} \text{ on } X_{reg}\}  \\
       & = & \# \ \text{locus}\{d\tilde{f_1}, d\tilde{f_2}\}\cap\text{locus}\{l_1,\ldots, l_{d}\}\\
       & = & \Sigma(\tilde{f}|_{X_{reg}})\cap H\\
        & = & \Sigma(\tilde{f})\cap H.\\
  \end{array}  
\end{equation*}
\end{proof}

In the next result, we consider $Ind_{X,0}\{\omega_j^{(i)}\}$ the GSV-index for collections of 1-forms defined by Ebeling and Gusein-Zade \cite{EGZLondon}, which is a generalization of the GSV-index defined in \cite{GSV}. 

\begin{definition}
Let $f:(X,0)\rightarrow(\mathbb{C}^2,0)$ be an $A$-finite holomorphic map-germ, where $(X,0)\subset(\mathbb{C}^n,0)$ is a 2-dimensional isolated complete intersection singularity (ICIS). Given $\mathcal{F}:(\mathcal{X},0)\rightarrow(\mathbb{C}\times\mathbb{C}^2,0)$, $\mathcal{F}(s,x)=(s,f_s(x))$, a stabilisation of $f$, we consider a small representative $\mathcal{F}:\mathcal{X}\rightarrow D\times B_{\varepsilon}$, where $D$ is an open neighbourhood of $0$ in $\mathbb{C}$ and $B_{\varepsilon}$ is a ball of radius $\varepsilon$ in $\mathbb{C}^2$. We define by $c(f_s)$ the number of cusps of $f_s:X_s\rightarrow B_{\varepsilon}$, for $s\in D\setminus\{0\}$.
\end{definition}

In \cite{BOT} the authors show that the number of cusps $c(f_s)$ is independent of $s\in D\setminus\{0\}$ and of the stabilisation $\mathcal{F}$.

\begin{proposition}
Let $f:(X,0)\rightarrow(\mathbb{C}^2,0), f(x)=(f_1(x),f_2(x))$, be an $A$-finite holomorphic map-germ, where $(X,0)\subset(\mathbb{C}^n,0)$ is a 2-dimensional ICIS. Consider the collection $\{\omega_j^{(i)}\}=\{\omega^{(1)}, \omega^{(2)}\}$ of 1-forms on $(\mathbb{C}^n,0)$, with an isolated special point at the origin, given by $\omega^{(1)}=\{df_1,df_2\}$ and $\{\omega^{(2)}\}=\{df_1,d\Delta\}$, where $\Delta$ is the determinant of the jacobian matrix of $f$. Then $$Ch_{X,0}\{\omega_j^{(i)}\}=c(f_s)-Ind_{X,0}\{\ell_j^{(i)}\},$$ where $c(f_s)$ is the number of cusps of a stable deformation $f_s$ in a stabilisation of $f$ and $\{\ell_j^{(i)}\}$ is a generic collection  of linear functions on $\mathbb{C}^n$. 
\end{proposition}

\begin{proof} Let us consider a stabilisation of $f$ as in Definition 3.2 of \cite{BOT}. We denote by $X_s$ a smoothing of $(X,0)$ and by $f_s:X_s\rightarrow \mathbb{C}^2$ a small representative of a stable map in the stabilisation. By Corollary 3 of \cite{Ebeling2007a}, for a generic collection $\{\ell_j^{(i)}\}$ of linear functions on $\mathbb{C}^n$, we have that $$Ch_{X,0}\{\omega_j^{(i)}\}= Ind_{X,0}\{\omega_j^{(i)}\}-Ind_{X,0}\{\ell_j^{(i)}\}.$$ Following \cite{EGZLondon}, $Ind_{X,0}\{\omega_j^{(i)}\}$ is equal to the sum of the indices of the singular points $q_1.\ldots, q_l$ of $\{\omega_j^{(i)}\}$ on the smoothing $X_s$ of $X$ in a neighbourhood of the origin. Since the local Chern obstruction coincides with the index on a smooth manifold and $f_s:X_s\rightarrow \mathbb{C}^2$ is stable, thus finitely determined, it follows from \cite[ Proposition 3.2]{brasselet2010euler} that 
\begin{equation*}
  \begin{array}{ccl}
     Ind_{X,0}\{\omega_j^{(i)}\}  &=& \displaystyle\sum_{i=1}^{\l}Ind_{X_s,q_i}\{{\omega}_j^{(i)}\} \\
     &=&  \displaystyle\sum_{i=1}^{\l}Ch_{X_s,q_i}\{{\omega}_j^{(i)}\} \\
        & = & \displaystyle\sum_{i=1}^{\l}c_i({f_s})\\
        & = & c(f_s),
  \end{array}  
\end{equation*} where $c_i$ is the number os cusps of $f_s$ close to $q_i$. 
\end{proof}

\bibliography{GRS} 

\begin{thebibliography}{10}

\bibitem{brasselet2010euler}
J.-P. Brasselet, N.~G. Grulha~Jr, and M.~A. Ruas.
\newblock The {E}uler obstruction and the {C}hern obstruction.
\newblock {\em Bulletin of the London Mathematical Society}, 42(6):1035--1043,
  2010.

\bibitem{BLS}
J.-P. Brasselet, D.~T. L{\^e}, and J.~Seade.
\newblock Euler obstruction and indices of vector fields.
\newblock {\em Topology}, 39(6):1193 -- 1208, 2000.

\bibitem{BMPS}
J.-P. Brasselet, D.~Massey, A.~Parameswaran, and J.~Seade.
\newblock Euler obstruction and defects of functions on singular varieties.
\newblock {\em Journal of the London Mathematical Society}, 70(1):59--76, 2004.

\bibitem{brasselet1981classes}
J.-P. Brasselet and M.-H. Schwartz.
\newblock Sur les classes de {C}hern d'un ensemble analytique complexe.
\newblock {\em Ast{\'e}risque}, 82:93--147, 1981.

\bibitem{milnor}
G.~Chudnovsky.
\newblock Singular points on complex hypersurfaces and multidimensional
  {S}chwarz lemma.
\newblock {\em S{\'e}minaire de Th{\'e}orie des Nombres, Paris}, 80, 1979.

\bibitem{DG}
N.~Dutertre and N.~G. Grulha~Jr.
\newblock L{\^e}--{G}reuel type formula for the {E}uler obstruction and
  applications.
\newblock {\em Advances in Mathematics}, 251:127--146, 2014.

\bibitem{DutertreGrulhaJofSing}
N.~Dutertre and N.~G. Grulha~Jr.
\newblock Some notes on the {E}uler obstruction of a function.
\newblock {\em J. Singul.}, 10:82--91, 2014.

\bibitem{EGZLondon}
W.~Ebeling and S.~Gusein-Zade.
\newblock Indices of vector fields or 1-forms and characteristic numbers.
\newblock {\em Bulletin of the London Mathematical Society}, 37:747--754, 2005.

\bibitem{ebeling2005radial}
W.~Ebeling and S.~Gusein-Zade.
\newblock Radial index and {E}uler obstruction of a 1-form on a singular
  variety.
\newblock {\em Geometriae Dedicata}, 113(1):231--241, 2005.

\bibitem{Ebeling2007a}
W.~Ebeling and S.~M. Gusein-Zade.
\newblock Chern obstructions for collections of 1-forms on singular varieties.
\newblock pages 557--564, 2007.

\bibitem{GaffneyGrulhaJofSing}
T.~Gaffney and N.~G. Grulha~Jr.
\newblock The multiplicity polar theorem, collections of 1-forms and chern
  numbers.
\newblock {\em J. Singul.}, 7:82--91, 2013.

\bibitem{GSV}
X.~G\'{o}mez-Mont, J.~Seade, and A.~Verjovsky.
\newblock The index of a holomorphic flow with an isolated singularity.
\newblock {\em Math. Ann.}, 291:737--751, 1991.

\bibitem{Grulha2008}
N.~G. Grulha~Jr.
\newblock L'obstruction d'{E}uler locale d'une application.
\newblock {\em Annales de la Facult\'e des sciences de Toulouse :
  Math\'ematiques}, 17(1):53--71, 2008.

\bibitem{Hamm}
H.~Hamm.
\newblock Lokale topologische {E}igenschaften komplexer r{\"a}ume.
\newblock {\em Mathematische Annalen}, 191(3):235--252, 1971.

\bibitem{iomdin1974complex}
I.~Iomdin.
\newblock Complex surfaces with a one-dimensional set of singularities.
\newblock {\em Siberian Mathematical Journal}, 15(5):748--762, 1974.

\bibitem{le1980ensembles}
D.~T. L{\^e}.
\newblock Ensembles analytiques complexes avec lieu singulier de dimension un
  (d’après {I}omdine).
\newblock {\em Seminaire sur les Singularit{\'e}s, Publications
  Matematiqu{\'e}s de l’Universit{\'e} Par{\i}s VII}, pages 87--95, 1980.

\bibitem{MacPherson}
R.~D. MacPherson.
\newblock Chern classes for singular algebraic varieties.
\newblock {\em Annals of Mathematics}, pages 423--432, 1974.

\bibitem{Ms2}
D.~B. Massey.
\newblock {\em Numerical control over complex analytic singularities}.
\newblock American Mathematical Soc.

\bibitem{Ms1}
D.~B. Massey.
\newblock Hypercohomology of {M}ilnor fibres.
\newblock {\em Topology}, 35(4):969--1003, 1996.

\bibitem{BOT}
J.~Nu{\~n}o-Ballesteros, B.~Oréfice-Okamoto, and J.~Tomazella.
\newblock Equisingularity of map germs from a surface to the plane.
\newblock {\em Collectanea Mathematica}, 69(1):65--81, 2018.

\bibitem{NOT}
J.~J. Nu{\~n}o-Ballesteros, B.~Orefice-Okamoto, and J.~N. Tomazella.
\newblock The vanishing {E}uler characteristic of an isolated determinantal
  singularity.
\newblock {\em Israel Journal of Mathematics}, 197(1):475--495, 2013.

\bibitem{ruas2014codimension}
M.~A.~S. Ruas and M.~S. Pereira.
\newblock Codimension two determinantal varieties with isolated singularities.
\newblock {\em Mathematica Scandinavica}, pages 161--172, 2014.

\bibitem{santana2019local}
H.~Santana.
\newblock Local topology of a deformation of a function-germ with a
  one-dimensional critical set, 2019.

\bibitem{santana2019brasselet}
H.~Santana.
\newblock Brasselet number and function-germs with a one-dimensional critical
  set.
\newblock {\em Bull Braz Math Soc, New Series}, 51:301--331, 2020.

\bibitem{STV}
J.~Seade, M.~Tib{\u{a}}r, and A.~Verjovsky.
\newblock Milnor numbers and {E}uler obstruction.
\newblock {\em Bulletin of the Brazilian Mathematical Society}, 36(2):275--283,
  2005.

\bibitem{Teissier1974}
B.~Teissier.
\newblock Cycles évanescents et conditions de {W}hitney.
\newblock {\em C. R. Acad. Sc. Paris}, 276(5):1051–1054, 1973.

\end{thebibliography}
\bibliographystyle{abbrv}

 \end{document}